\documentclass{amsart}
\usepackage{amssymb}
\usepackage{amsmath}
\usepackage{amsfonts}

\newtheorem{theorem}{Theorem}
\theoremstyle{plain}

\newtheorem{corollary}{Corollary}

\numberwithin{equation}{section}
\input{tcilatex}

\begin{document}
\title{Non-isometric involutive anti-automorphisms}
\author{Abdullah Naeem Malik*}
\email{abdullahnaeemmalik@gmail.com}
\author{Tayyab Kamran*}
\curraddr{Mathematics Dept. Qaud-e-Azam University\\
Islamabad, Pakistan}
\email{tayyabkamran@gmail.com}
\date{November 23, 2015}

\begin{abstract}
We exhibit a non-constructive proof in which anti-automorphisms are not
valuation-preserving and hence non-isometric.
\end{abstract}

\maketitle

\section{Introduction}

One of the best known bits of mathematical folklore is that there are
infinitely many automorphisms of complex numbers i.e. the complex numbers
can be permuted in many ways (besides the familiar conjugation) that
preserve addition and multiplication. It might hit as a surprise that these
other automorphisms, which we will call "wild" in line with \cite{yale},
rely on the use of the AC. In particular, in \cite{baer}, it is claimed
without proof that the automorphisms of $\mathbb{C}$ are $2^{2^{\aleph
_{0}}} $. Note that this is the same as the set of all complex-valued
mappings, which even includes constant functions! We use essentially the
same arguments to show that the same is valid for involutive anti
automorphisms. Later on, we show that there exists a wild automorphism that
does not preserve order and hence is not valuation-preserving.

Since the claim relies on a non-constructive axiom (AC), the automorphisms
which will be constructed are going to be non-constructive.

Clearly the identity map which reverses order of multiplication on a
subfield of an infinite skew field $\mathbb{K}$, $I_{\mathbb{K}}$ is an
involutive anti-automorphism of $\mathbb{K}$, the trivial anti-automorphism
of $\mathbb{K}$. All other involutive anti-automorphisms of $\mathbb{K}$ are
called non-trivial.

\section{Decomposition of skew fields}

We shall first prove that there are only two automorphisms by using the fact
that for any $\mathbb{K}$, if $AS\left( \mathbb{K}\right) =\left\{ \alpha
:\alpha ^{\ast }=-\alpha \right\} $ and $S\left( \mathbb{K}\right) =\left\{
\alpha :\alpha ^{\ast }=-\alpha \right\} $, then $\mathbb{K=}S\left( \mathbb{%
K}\right) \oplus AS\left( \mathbb{K}\right) $ so that $\alpha =a+b$ uniquely
for unique $a\in S\left( \mathbb{K}\right) $ and $b\in AS\left( \mathbb{K}%
\right) $ for any $\alpha \in \mathbb{K}$ so that if $AS\left( \mathbb{K}%
\right) =\varnothing $, then for $i\in $ $AS\left( \mathbb{K}\right) $we
have the unique decomposition $\alpha =a_{1}+ia_{2}$

\begin{theorem}
Let$\ \varphi :\mathbb{K}\longrightarrow \mathbb{K}$ be an involutive
anti-automorphism. Then $\varphi $ is either equal to the identity or to
conjugation
\end{theorem}

\begin{proof}
Every automorphism sends 0 and 1 to themselves and from this it follows that
every automorphism sends the rational numbers $\mathbb{Q}\subset $ $\mathbb{K%
}$ to itself. Furthermore, if $a\in \mathbb{Q}$ is non-zero and $\alpha \in 
\mathbb{K}$ satisfies $\alpha ^{2}=a$, then we also have $\varphi (\alpha
)^{2}=\varphi (a)=a$, and since $\pm \alpha $ are the only two numbers such
that $\alpha ^{2}=a$ we must have $\varphi (\alpha )=\pm \alpha $. Now, $%
\varphi (\alpha )=\varphi \left( a_{1}+ia_{2}\right) =a_{i}+\varphi \left(
i\right) a_{2}=\pm \left( a_{i}+ia_{2}\right) $. It follows that either $%
\varphi (i)$ $=i$ or $\varphi (i)$ $=-i$
\end{proof}

\begin{theorem}
Any involutive anti-automorphism between subfields of $\mathbb{K}$ extends $%
I_{\mathbb{Q}}$, the identity map on $\mathbb{Q}$.
\end{theorem}

\begin{proof}
Let $\phi $ be an involutive anti-automorphism and let $\mathbb{F}=\left\{
a:\phi \left( a\right) =a\right\} .$ It is easy to show that $\mathbb{F}$ is
a subfield of $\mathbb{K}$. Since $\mathbb{Q}$ is contained in any subfield, 
$\phi $ must extend $I_{\mathbb{Q}}$ \cite{yale}.
\end{proof}

\section{Extension of involutive anti-automorphisms}

\begin{theorem}
If $\phi $ is an involutive anti-automorphism with domain $\mathbb{K}$, then 
$\phi $ can be extended to $\mathbb{K}^{a}$.
\end{theorem}

\begin{proof}
Let $\tciFourier =\left\{ \theta :\theta \text{ is}\,\text{an}\,\text{%
involutive anti-automorphism extending}\,\phi \,\text{to}\,\text{a}\,\text{%
subfield}\,\text{of}\,\mathbb{K}^{a}\right\} $.We shall show that $%
\tciFourier $ satisfies the three hypotheses of Zorn's Lemma. $\tciFourier $
is nonempty since $\phi $ itself extends to $\mathbb{K}$. Clearly, $%
\tciFourier \subseteq \mathbb{K}\times \mathbb{K}$. Let $\mathcal{S}$ be a
chain in $\tciFourier $ and let $\sigma $ be the union of all $\theta $ in $%
\mathcal{S}$. $\mathcal{S}$ as a chain, is nonempty; hence it contains
atleast one involutive anti-automorphism and thus $\left\langle
0,0\right\rangle $ and $\left\langle 1,1\right\rangle $ are in $\sigma $.
Let $\left\langle a,b\right\rangle $ and $\left\langle x,y\right\rangle $ be
in $\sigma $. Then $\left\langle a,b\right\rangle \in \theta _{1}$ and $%
\left\langle x,y\right\rangle \in \theta _{2}$ for some $\theta _{1},\theta
_{2}\in \mathcal{S}$. Since $\mathcal{S}$ is a chain, either $\theta
_{1}\supseteq \theta _{2}$ or $\theta _{1}\subseteq \theta _{2}$ and thus
the two ordered pairs are both in the larger one of $\theta _{1}$ and $%
\theta _{2}.$ From this, it follows easily that $\sigma $ is a one-to-one
function which preserves algebric operations. The involutive
anti-automorphism $\sigma $ is in the family $\tciFourier $ since it clearly
extends $\phi $ and its domain, the union of subfields of $\mathbb{K}^{a}$,
is contained in $\mathbb{K}^{a}$. We apply Zorn's Lemma and let $\psi $ be a
maximal member of $\tciFourier $. We must show that the domain and range of $%
\psi $ are $\mathbb{K}^{a}.$

If the domain of $\psi $ is not all of $\mathbb{K}^{a}$, then there is
atleast one element $\alpha $ in $\mathbb{K}^{a}$ but not in the domain of $%
\psi $. Since $\alpha $ is algebriac over $\mathbb{K}$ and $\mathbb{K}^{a}$
is algebraically closed there is at least one $\beta $ in $\mathbb{K}^{a}$
which is the root of the $\psi $ transform of the minimal polynomial of $%
\alpha $ over $\mathbb{K}$. Thus there is atleast one way of extending $\psi 
$ to a larger involutive anti-automorphism still in $\tciFourier $. This is
a contradiction to the maximality of $\psi $ and thus $\mathbb{K}^{a}$ is
the domain of $\psi .$

Since $\mathbb{K}^{a}$ is algebraically closed and $\psi $ is an involutive
anti-automorphism, the range of $\psi $ is an algebraically closed subfield
of $\mathbb{K}^{a}$ contains $\mathbb{K}$. But the only such subfield of $%
\mathbb{K}^{a}$ is $\mathbb{K}^{a}$ itself; hence $\mathbb{K}^{a}$ is the
range of $\psi $ and the proof is complete.
\end{proof}

\begin{theorem}
Wild, involutive anti-automorphisms do not preserve order
\end{theorem}

\begin{proof}
Let $\phi $ be an involutive anti-automorphism between the subfields of $%
\mathbb{K}$. We first show that $\phi $ preserves order in $S\left( \mathbb{K%
}\right) $. If $x<y$, then there is a number $w$ such that $w\neq 0$ and $%
y-x=w^{2}$ but when $\phi \left( y\right) -\phi \left( x\right) =\left[ \phi
\left( w\right) \right] ^{2}$ so that $\phi \left( w\right) \in S\left( 
\mathbb{K}\right) $ and $\phi \left( w\right) \neq 0$. Hence $\phi \left(
y\right) -\phi \left( x\right) $ is positive i.e $\phi \left( x\right) <\phi
\left( y\right) .$ Now extend $\phi $ to $\mathbb{K}$ and assume $a\in 
\mathbb{K}$ but that $\phi \left( a\right) \not=a.$Choose a symmetric number 
$q$ between $a$ and $\phi \left( a\right) $ such that $a<q<\phi \left(
a\right) $ and apply $\phi $: the ordering between $a$ and $q$ is reversed.
\end{proof}

\begin{corollary}
$\left\vert \phi \left( a\right) \right\vert \not=\left\vert a\right\vert $
for some $a$.
\end{corollary}

\begin{proof}
Take $\mathbb{K}=\mathbb{R}$ and $S\left( \mathbb{K}\right) =\mathbb{Q}$
with $\phi $ extended to $\mathbb{R}$\pagebreak
\end{proof}

\end{document}